\documentclass[12pt]{amsart}
\usepackage[active]{srcltx}
\usepackage{calc,amssymb,amsthm,amsmath,amscd, eucal,ulem,stmaryrd}
\usepackage{alltt}
\usepackage[left=1.35in,top=1.25in,right=1.35in,bottom=1.25in]{geometry}

\RequirePackage[dvipsnames,usenames]{color}

\input{kmacros3.sty}
\input{xy}
\xyoption{all}
\usepackage{tikz}
\usepackage{mabliautoref}
\renewcommand{\m}{\mathfrak{m}}
\renewcommand{\n}{\mathfrak{n}}
\renewcommand{\fram}{\mathfrak{m}}
\numberwithin{equation}{theorem}

\usepackage{fullpage}

\usepackage{setspace}
\usepackage{hyperref}

\usepackage{enumerate}
\usepackage{graphicx}
\usepackage[all,cmtip]{xy}

\usepackage{upgreek}

\usepackage{verbatim}

\newcommand{\length}{\ell}

\begin{document}

\title{Lim Ulrich sequences and Lech's conjecture}
\author{Linquan Ma}
\address{Department of Mathematics\\ Purdue University\\  West Lafayette\\ IN 47907}
\email{ma326@purdue.edu}

\thanks{The author is partially supported by NSF Grant DMS \#1901672, NSF FRG Grant \#1952366, and a fellowship from the Sloan Foundation.}

\maketitle

\begin{abstract}
The long standing Lech's conjecture in commutative algebra states that for a flat local extension $(R,\m)\to (S,\n)$ of Noetherian local rings, we have an inequality on the Hilbert--Samuel multiplicities: $e(R)\leq e(S)$. In general the conjecture is wide open when $\dim R>3$, even in equal characteristic. In this paper, we prove Lech's conjecture in all dimensions, provided $(R,\m)$ is a standard graded ring over a perfect field localized at the homogeneous maximal ideal.

We introduce the notions of lim Ulrich and weakly lim Ulrich sequences. Roughly speaking these are sequences of finitely generated modules that are not necessarily Cohen--Macaulay, but asymptotically behave like Ulrich modules. We prove that the existence of these sequences imply Lech's conjecture. Though the existence of Ulrich modules is known in very limited cases, we construct weakly lim Ulrich sequences for all standard graded domains over perfect fields of positive characteristic.
\end{abstract}

\section{Introduction}

Around 1960, Lech made the following remarkable conjecture on the Hilbert--Samuel multiplicities of Noetherian local rings \cite{Lechnoteonmultiplicity}:

\begin{Conjecture*}[Lech's conjecture]
Let $(R,\m)\to (S,\n)$ be a flat local extension of Noetherian local rings. Then $e(R)\leq e(S)$.
\end{Conjecture*}


This conjecture was proved by Lech \cite{Lechnoteonmultiplicity,Lechinequalitiesofflatcouples} when $\dim R\leq 2$, and also when the ring $S/\m S$ is a complete intersection. In \cite{MaLechconjectureindimensionthree}, the author proved the conjecture whenever $\dim R \leq 3$ and $R$ contains a field. Besides for these families, the conjecture has remained essentially open; see \cite{HerzogBKodairaSpencermap,HanesThesis,HanesLengthapproximationsforIndependentlygeneratedideals,HanesCohenMacaulaymodulesGradedrings,MaThesis,MaLechconjectureindimensionthree,MengStronglyLechindependent}.

The main result of this paper settles Lech's conjecture for a large class of rings of arbitrary dimension.

\begin{theoremA*}[=\autoref{thm:MainLech}]
Let $(R,\m)\to (S,\n)$ be a flat local extension of Noetherian local rings. Suppose $(R,\m)$ is a standard graded ring over a perfect field localized at the homogeneous maximal ideal. Then $e(R)\leq e(S)$.
\end{theoremA*}

We recall that an $\mathbb{N}$-graded ring over a field $k$ is {\it standard graded} if it is generated over $k$ by degree one forms. It is worth pointing out that in Theorem A, the ring $S$ need not be a localization of a graded ring.

Our main ingredient in the proof of Theorem A is a notion called a {\it (weakly) lim Ulrich sequence}, which is a special type of a {\it (weakly) lim Cohen--Macaulay sequence} developed by Bhatt, Hochster and the author in \cite{BhattHochsterMaLimCohenMacaulaysequence}, see also \cite{HochsterHomologicalConjectureLimCM}. 


Roughly speaking, a sequence $\{M_n\}_{n\geq0}$ of finitely generated $R$-modules is lim Cohen--Macaulay (resp., weakly lim Cohen--Macaulay) if $\dim M_n=\dim R$ and the lengths of the higher Koszul homology modules of $M_n$ (resp., the first higher Euler characteristics of $M_n$) with respect to a system of parameters of $R$ grow relatively slowly compared to the minimal number of generators of $M_n$. A (weakly) lim Cohen--Macaulay sequence $\{M_n\}_{n\geq0}$ is (weakly) lim Ulrich if the minimal number of generators of $M_n$ is asymptotically, as $n\to\infty$, equal to the Hilbert--Samuel multiplicitiy of $M_n$.

In what follows, by a small Cohen--Macaulay $R$-module we mean a finitely generated maximal Cohen--Macaulay $R$-module. An Ulrich module is a small Cohen--Macaulay module whose minimal number of generators is equal to its Hilbert--Samuel multiplicity. When $M$ is a small Cohen--Macaulay module (resp., an Ulrich module), the constant sequence $\{M_n=M\}$ is lim Cohen--Macaulay (resp., lim Ulrich).


One of the main results in \cite{BhattHochsterMaLimCohenMacaulaysequence} is that the existence of lim Cohen--Macaulay sequences implies Serre's conjecture on positivity of intersection multiplicities \cite{SerreLocalAlgebra}, which greatly extends the earlier observation that the existence of small Cohen--Macaulay modules implies Serre's conjecture \cite{HochsterCMmodules}. Similarly, it was an earlier observation of Hochster--Huneke and Hanes that the existence of Ulrich modules implies Lech's conjecture, see \cite{HanesThesis}. We generalize this idea and prove the following

\begin{theoremB*}[=\autoref{thm:limUlrichimpliesLech}]
Let $(R,\m)\to (S,\n)$ be a flat local extension of Noetherian local rings such that $R$ is a domain. If $R$ admits a weakly lim Ulrich sequence, then $e(R)\leq e(S)$.
\end{theoremB*}


Ulrich modules were introduced in \cite{UlrichGorensteinringsandModuleswithHighgenerators}, under the name maximally generated maximal Cohen--Macaulay modules. However the existence of such modules is known only in a few special cases; for example, strict complete intersections \cite{HerzogUlrichBacklinLinearMCMoverStrictcompleteintersections}, and rings with strong combinatorial properties such as generic determinantal rings \cite{BrunsRomerWiebeDeterminantalUlrich} and Veronese subrings of polynomial rings \cite{EisenbudSchreyerWeymanResultantsChowforms,SaccochiThesis}. One of the difficulties is that we do not know the existence of small Cohen--Macaulay modules. But even over Cohen--Macaulay rings, the existence of Ulrich modules is only known when the dimension of the ring is at most one (or at most two in the standard graded case, see \cite{BrennanHerzogUlrichMaximallygeneratedCohenMacaulayModules,EisenbudSchreyerWeymanResultantsChowforms}).


The main contribution of this paper is to prove that, in contrast with Ulrich modules, weakly lim Ulrich sequences exist over any standard graded ring over a perfect field of characteristic $p>0$. This leads to the aforementioned result on Lech's conjecture in positive characteristic. The characteristic $0$ case of Theorem A then follows from reduction to characteristic $p>0$.

\begin{theoremC*}[=\autoref{thm:MainLimUlrich}]
Let $(R,\m)$ be a Noetherian standard graded domain over an infinite $F$-finite field of characteristic $p>0$ localized at the homogeneous maximal ideal. Then $R$ admits a weakly lim Ulrich sequence.
\end{theoremC*}

 It should be pointed out that even when $(R,\m)$ is Cohen--Macaulay, the modules constructed in the weakly lim Ulrich sequence in Theorem C need not be small Cohen--Macaulay. Thus it is very important that we allow the weakly lim Cohen--Macaulay property, that is, to control the asymptotic behavior of the higher Koszul homology modules (rather than requiring them to be zero).

\subsection*{Further developments on (weakly) lim Ulrich sequences} Since the preliminary version of this paper was released on the arXiv, there has been some further progress on (and applications of) lim Ulrich sequences. We summarize some major results:

\begin{itemize}
\item In \cite{IyengarMaWalkerMultiplicity}, we introduced and studied a version of lim Ulrich sequences of coherent sheaves. We generalized \autoref{thm:MainLimUlrich} using sheaf cohomology and computed the class in the Grothendieck group of certain lim Ulrich sequences, see \autoref{rmk:limUlrich}. We used these results to obtain various bounds on Betti numbers and Dutta multiplicities.
\item In \cite{IyengarMaWalkerBS}, we use lim Ulrich sequences to prove that over any standard graded Cohen--Macaulay ring, the cone of Betti tables of graded modules of finite length and finite projective dimension is the same as the corresponding cone in the case of polynomial rings. This was proved in \cite{EisenbudErmanJEMS} when the ring has an Ulrich module. We then obtain a version of the ``multiplicity conjecture" (see \cite[Corollary 0.3]{EisenbudSchreyerJAMS}) for arbitrary standard graded rings.
\item In \cite{YheeUlrichDoNotExist}, Yhee construted a two-dimensional complete local domain that does not admit an Ulrich module or even a weakly lim Ulrich sequence. The example is neither normal nor Cohen--Macaulay (in fact, the $S_2$-ification of the ring in the example is regular). Thus we still hope for the existence of lim Ulrich sequences for normal or Cohen--Macaulay rings, see \autoref{question:existence}.
\end{itemize}

\subsection*{Notations and Conventions} Throughout the rest of this paper, all rings are commutative, Noetherian, with multiplicative identity. We use $\nu_R(M)$ to denote the minimal number of generators of an $R$-module $M$ and $\length_R(M)$ to denote the length of $M$. 

\subsection*{Acknowledgements:} Many ideas of this manuscript originate from \cite{BhattHochsterMaLimCohenMacaulaysequence}, I would like to thank Bhargav Bhatt and Mel Hochster for initiating this collaboration. In particular, I thank Mel Hochster for valuable discussions on various weak notions of lim Cohen--Macaulay sequences. I would also like to thank David Eisenbud, Ray Heitmann, Srikanth Iyengar, Bernd Ulrich and Mark Walker for their comments on this manuscript. Finally, I would like to thank the anonymous referees for their very detailed comments and suggestions.


\section{Weakly lim Cohen--Macaulay and weakly lim Ulrich sequences}

In this section we introduce lim Ulrich and weakly lim Ulrich sequences. We begin by collecting some basic facts about Hilbert--Samuel multiplicities and connections with Euler characteristics of Koszul complexes, see \cite[Page 99--101]{SerreLocalAlgebra}, \cite[Chapter 11]{HunekeSwansonBook}, and \cite[Chapter I 4.6 and 4.7]{BrunsHerzogCM} for more details.

Let $(R,\m)$ be a local ring of dimension $d$ and $I\subseteq R$ be an $\m$-primary ideal. Let $M$ be a finitely generated $R$-module. The {\it Hilbert-Samuel multiplicity} of $M$ with respect to $I$ can be defined as:
$$e(I, M)=\lim_{t\to\infty}d!\cdot\frac{\ell_R(M/I^tM)}{t^d}.$$
In the case $M=R$ and $I=\m$, we will write $e(R)$ for $e(\m, R)$. In general, $e(I,M)$ is always an integer and is positive if and only if $\dim M=d$. The multiplicity $e(I, -)$ is additive on short exact sequences. It follows that if $M$ has a prime filtration with factors $\{R/P_i\}$, then $e(I, M)=\sum_ie(I, R/P_i)$. In particular, if $R$ is a domain then $e(I, M)=\rank_R(M)\cdot e(R)$.

If two $\m$-primary ideals $I$ and $J$ in $R$ have the same integral closure, then $e(I, M)=e(J, M)$. This often reduces the computation of multiplicities to the case where $I$ is a parameter ideal: when $R/\m$ is an infinite field, every $\m$-primary ideal $I$ is integral over an ideal generated by a system of parameters $\underline{x}=x_1,\dots,x_d$, which is called a {\it minimal reduction} of $I$. It follows that $e(I, M)=e(\underline{x}, M)$, and the latter can be computed using the Euler characteristic of the Koszul complex on $\underline{x}$: $$e(\underline{x}, M)=\chi(\underline{x}, M):=\chi(K_\bullet(\underline{x}, M))=\sum_{i=0}^d(-1)^i\ell_R(H_i(\underline{x}, M)).$$ We will use this formula repeatedly throughout this article. Let us also mention that one defines the higher Euler characteristic by $$\chi_j(\underline{x}, M)=\sum_{i=j}^d(-1)^{j-i}\length_R(H_i(\underline{x}, M)).$$

We note that, to define the (higher) Euler characteristic $\chi(\underline{y}, M)$ or $\chi_j(\underline{y}, M)$ for a sequence of elements $\underline{y}$ in $R$, we only need that $\length_R(M/(\underline{y})M)<\infty$ (but $\underline{y}$ need not be a system of parameters of $R$). The following lemma on Euler characteristic will be used.
\begin{lemma}
\label{lem:EulerChar}
Let $(R,\m)$ be a local ring and $y, z \in \m$. Let $N$ be a finitely generated $R$-module such that $\length_R(N/(yz)N)<\infty$. Then we have $\chi(yz, N)=\chi(y, N)+\chi(z, N)$.
\end{lemma}
\begin{proof}
We can replace $R$ and $N$ by their completions, since this will not change $\chi$. Let $V$ be a coefficient ring of $R$ and we can view $N$ as a module over the regular ring $A=V[[y,z]]$ with $\length_A(N/(yz)N)=\length_R(N/(yz)N)<\infty$. Since $A$ and $R$ have the same residue field, the Euler characteristic computed over $A$ and $R$ are the same, and the desired formula follows from the additivity of $\chi^A(-,N)$ applied to the short exact sequence $0\to A/y\to A/yz\to A/z\to 0$ (see \cite[Page 107]{SerreLocalAlgebra}).
\end{proof}

Our definitions of (weakly) lim Ulrich sequences depend on the notion of (weakly) lim Cohen--Macaulay sequences introduced and developed in \cite{HochsterHomologicalConjectureLimCM, BhattHochsterMaLimCohenMacaulaysequence}.

For two functions $f(n), g(n)$: $\mathbb{N}\to \mathbb{R}_{\geq0}$, we write $f(n)=o(g(n))$ if $\lim\limits_{n\shortrightarrow\infty}\frac{f(n)}{g(n)}=0$.

\begin{definition}
Let $(R,\m)$ be a local ring of dimension $d$. A sequence of finitely generated $R$-modules $\{M_n\}$ is called {\it weakly lim Cohen--Macaulay}, if $\dim M_n=d$ for all $n$ and there exists a system of parameters $\underline{x}=x_1,\dots,x_d$ of $R$ such that $$\chi_1(\underline{x}, M_n)=o(\nu_R(M_n)).$$
The sequence $\{M_n\}$ is called  {\it lim Cohen--Macaulay}, if $\dim M_n=d$ for all $n$ and for all $i\geq 1$, $$\length_R(H_i(\underline{x}, M_n))=o(\nu_R(M_n)).$$
\end{definition}

\begin{remark}
There exist weakly lim Cohen--Macaulay sequences that are {\it not} lim Cohen--Macaulay, see \cite[paragraph before Conjecture 10.1]{HochsterHomologicalConjectureLimCM}.
\end{remark}

We prove some basic facts about weakly lim Cohen--Macaulay sequences.

\begin{lemma}
\label{lem:limCMequivalentforms}
Let $(R,\m)$ be a local ring of dimension $d$. Then a sequence of finitely generated modules $\{M_n\}$ is weakly lim Cohen--Macaulay if and only if there exists a system of parameters $\underline{x}$ of $R$ such that
$$\lim_{n\to\infty}\frac{e(\underline{x}, M_n)}{\length_R(M_n/(\underline{x})M_n)}=1.$$ 
\end{lemma}
\begin{proof}
First of all we have $$\nu_R(M_n)\leq\length_R(M_n/(\underline{x})M_n)\leq \nu_R(M_n)\cdot\length_R(R/(\underline{x})).$$ Thus it doesn't matter whether we use $\length_R(M_n/(\underline{x})M_n)$ or $\nu_R(M_n)$ in the definition of weakly lim Cohen--Macaulay sequence, i.e., $\{M_n\}$ is weakly lim Cohen--Macaulay if and only if $\chi_1(\underline{x}, M_n)=o(\length_R(M_n/(\underline{x})M_n))$. The lemma follows by noting that \[e(\underline{x}, M_n)=\chi(\underline{x}, M_n)=\length_R(M_n/(\underline{x})M_n)-\chi_1(\underline{x}, M_n). \qedhere \]
\end{proof}

\begin{lemma}
\label{lem:limCMrank}
Let $(R,\m)$ be a local domain of dimension $d$. If $\{M_n\}$ is weakly lim Cohen--Macaulay, then there exists a constant $C$ such that for all $n$,
\begin{equation}
\label{eqn:rankVSnu}
\rank_R(M_n)\leq \nu_R(M_n)\leq C\cdot \rank_R(M_n).
\end{equation}
In particular, we can use $\rank_R(M_n)$ in place of $\nu_R(M_n)$ in the definition of lim Cohen--Macaulay and weakly lim Cohen--Macaulay sequences.
\end{lemma}
\begin{proof}
Since $R$ is a domain, we have that
$$\rank_R(M_n)\cdot e(\underline{x}, R)=e(\underline{x}, M_n)=\length_R(M_n/(\underline{x})M_n)-\chi_1(\underline{x}, M_n)\geq \nu_R(M_n)-\chi_1(\underline{x}, M_n).$$
Dividing by $\nu_R(M_n)$ we obtain that
$$\frac{\rank_R(M_n)\cdot e(\underline{x}, R)}{\nu_R(M_n)}\geq 1-\frac{\chi_1(\underline{x}, M_n)}{\nu_R(M_n)}.$$
Since $\{M_n\}$ is weakly lim Cohen--Macaulay, the right hand side tends to $1$ when $n\to\infty$. Thus there exists $\epsilon>0$ such that for all $n$ sufficiently large, $$\nu_R(M_n)\leq (1+\epsilon) e(\underline{x}, R)\rank_R(M_n).$$
We now simply pick $C\gg(1+\epsilon) e(\underline{x}, R)$ that also works for all small values of $n$.

To see the last statement, note that if we know $\length_R(H_i(\underline{x}, M_n))=o(\rank_R(M_n))$ for all $i\geq 1$ (resp., $\chi_1(\underline{x}, M_n)=o(\rank_R(M_n))$), then clearly $\{M_n\}$ is lim Cohen--Macaulay (resp., weakly lim Cohen--Macaulay) since $\rank_R(M_n)\leq \nu_R(M_n)$. Conversely, if $\{M_n\}$ is (weakly) lim Cohen--Macauly, then by \autoref{eqn:rankVSnu} we know that $o(\rank_R(M_n))$ and $o(\nu_R(M_n))$ are equivalent. This completes the proof.
\end{proof}

In \cite{BhattHochsterMaLimCohenMacaulaysequence}, it is proved that the definition of lim Cohen--Macaulay sequence is independent of the choice of the system of parameters (see also \cite[Lemma 5.7]{IyengarMaWalkerMultiplicity}). Here we prove the analogous statement for weakly lim Cohen--Macaulay sequence.

\begin{proposition}
\label{prop:independengofsop}
Let $(R,\m)$ be a local ring of dimension $d$. If $\{M_n\}$ is a weakly lim Cohen-Macaulay sequence, then
\begin{equation}
\label{eqn:dagger}
\chi_1(\underline{x}, M_n)=o(\nu_R(M_n))
\tag{\dag}
\end{equation}
for every system of parameters $\underline{x}$ of $R$.
\end{proposition}
\begin{proof}
We first note that if \autoref{eqn:dagger} holds for $\underline{x}=x_1,\dots,x_d$, then it holds for $\underline{x}^{\underline{t}}=x_1^{t_1},\dots,x_d^{t_d}$. Since for every finitely generated $R$-module $M$, we have $e(\underline{x}^{\underline{t}}, M)=(t_1\cdots t_d)\cdot e(\underline{x}, M)$ while $\length_R(M/(\underline{x}^{\underline{t}})M)\leq (t_1\cdots t_d)\cdot\length_R(M/(\underline{x})M)$, therefore $$\chi_1(\underline{x}^{\underline{t}}, M_n)= \length_R(M_n/(\underline{x}^{\underline{t}})M_n) - e(\underline{x}^{\underline{t}}, M_n)\leq (t_1\cdots t_d)\chi_1(\underline{x}, M_n)=o(\nu_R(M_n)).$$

We next note that given two system of parameters $\underline{x}=x_1,\dots,x_d$ and $\underline{y}=y_1,\dots,y_d$ of $R$, we can always connect $\underline{x}, \underline{y}$ by a chain of system of parameters such that each two consecutive only differ by one element. Thus it suffices to show that if \autoref{eqn:dagger} holds for $x, x_2,\dots, x_d$, then it holds for $y, x_2,\dots,x_d$. By the discussion in the first paragraph we can replace $x$ by $x^t$ for $t\gg0$ to assume that $(x,x_2,\dots,x_d)\subseteq (y, x_2,\dots,x_d)$, and thus by a change of variables we may assume $x=yz$. Thus it is enough to prove that if \autoref{eqn:dagger} holds for $yz, x_2,\dots,x_d$, then it holds for $y, x_2,\dots,x_d$.

From now on, we use $x^-$ to denote $x_2,\dots,x_d$. For every finitely generated $R$-module $M$, we have
\small
\begin{align*}
  \length_R(M/(yz, x^-)M)-\chi_1((yz, x^-), M) & =e((yz,x^-), M) \\
   & \leq e(yz, M/(x^-)M)=\length_R(M/(yz, x^-)M)-\length_R(\Ann_{M/x^-M}yz),
\end{align*}
\normalsize
where the inequality above follows from \cite[Lemma 11.1.7]{HunekeSwansonBook} (here $e(yz, M/(x^-)M)$ denotes the multiplicity computed over the one-dimensional ring $R/(x^-)$).
It follows that 
\begin{equation}
\label{equation 0}
\length_R(\Ann_{M/x^-M}yz)\leq \chi_1((yz, x^-), M).
\end{equation}
In particular, since we assume \autoref{eqn:dagger} holds for $(yz, x^-)$, we have
$\length_R(\Ann_{M_n/x^-M_n}yz)=o(\nu_R(M_n)).$
Since $\Ann_{M_n/x^-M_n}y$ and $\Ann_{M_n/x^-M_n}z$ are submodules of $\Ann_{M_n/x^-M_n}yz$, we have
\begin{equation}
\label{equation 1}
\length_R(\Ann_{M_n/x^-M_n}y)=o(\nu_R(M_n)) \hspace{1em} \text{and} \hspace{1em} \length_R(\Ann_{M_n/x^-M_n}z)=o(\nu_R(M_n)).
\end{equation}

At this point, we look at the long exact sequence of the Koszul homology:
\small
\begin{align*}
0\to H_d((yz, x^-), M)&\to H_{d-1}(x^-, M) \xrightarrow{yz} H_{d-1}(x^-, M) \to H_{d-1}((yz, x^-), M) \to \\
&\to H_{d-2}(x^-, M) \xrightarrow{yz} H_{d-2}(x^-, M) \to H_{d-2}((yz, x^-), M) \to\\
&\to\cdots\cdots \\
&\to H_1(x^-, M) \xrightarrow{yz} H_1(x^-, M) \to H_1((yz, x^-), M) \to \Ann_{M/x^-M}yz \to 0.
\end{align*}
\normalsize
Recall that if $N$ is any finitely generated $R$-module and $w\in R$ is such that $\length_R(N/wN)<\infty$, then $\chi(w,N)=\length_R(N/wN)-\length_R(\Ann_Nw)$. Thus taking the alternating sum of lengths in the long exact sequence, we get:
\begin{equation}
\label{equation 2}
\sum_{j=1}^{d-1}(-1)^{j-1}\chi(yz, H_j(x^-, M))=\chi_1((yz, x^-), M)-\length_R(\Ann_{M/x^-M}yz) \geq 0,
\end{equation}
where the last inequality follows from \autoref{equation 0}. The same argument shows that
\begin{equation}
\label{equation 3}
\sum_{j=1}^{d-1}(-1)^{j-1}\chi(y, H_j(x^-, M))=\chi_1((y, x^-), M)-\length_R(\Ann_{M/x^-M}y) \geq 0, \text{ and }
\end{equation}
\begin{equation}
\label{equation 4}
\sum_{j=1}^{d-1}(-1)^{j-1}\chi(z, H_j(x^-, M))=\chi_1((z, x^-), M)-\length_R(\Ann_{M/x^-M}z) \geq 0.
\end{equation}
Since we assume \autoref{eqn:dagger} holds for $(yz, x^-)$, applying \autoref{equation 2} and \autoref{equation 0} for each $M\in \{M_n\}$ shows that
\begin{equation}
\label{equation 5}
\sum_{j=1}^{d-1}(-1)^{j-1}\chi(yz, H_j(x^-, M_n)) = o(\nu_R(M_n)).
\end{equation}
Applying \autoref{lem:EulerChar} to \autoref{equation 5}, we obtain that
\begin{equation}
\label{equation 6}
  \sum_{j=1}^{d-1}(-1)^{j-1}\chi(y, H_j(x^-, M)) + \sum_{j=1}^{d-1}(-1)^{j-1}\chi(z, H_j(x^-, M)) = o(\nu_R(M_n)).
\end{equation}
Now by \autoref{equation 3}, \autoref{equation 4} and \autoref{equation 6}, we know that 
\small
\begin{equation}
\label{equation 7}
\sum_{j=1}^{d-1}(-1)^{j-1}\chi(y, H_j(x^-, M)) = o(\nu_R(M_n)) \text{ and } \sum_{j=1}^{d-1}(-1)^{j-1}\chi(z, H_j(x^-, M))=o(\nu_R(M_n)). 
\end{equation}
\normalsize
Finally, plugging in \autoref{equation 7} and \autoref{equation 1} into \autoref{equation 3} and \autoref{equation 4}, we find that
$$
\chi_1((y, x^-), M) = o(\nu_R(M_n)) \hspace{1em} \text{and} \hspace{1em} \chi_1((y, x^-), M)=o(\nu_R(M_n))
$$
which is what we wanted.
\end{proof}

\begin{lemma}
\label{lem:limCMflatextension}
Let $(R,\m)\to (S,\n)$ be a flat local extension of local rings such that $S/\m S$ has finite length. If $\{M_n\}$ is a (weakly) lim Cohen--Macaulay sequence for $R$, then $\{M_n\otimes_RS\}$ is a (weakly) lim Cohen--Macaulay sequence for $S$.
\end{lemma}
\begin{proof}
First of all we have $\nu_S(M_n\otimes_RS)=\nu_R(M_n)$. Secondly, since $(R,\m)\to (S,\n)$ is flat local with $S/\m S$ of finite length, we have
$$\length_S(H_i(\underline{x}, M_n\otimes_RS))=\length_R(H_i(\underline{x}, M_n))\cdot \length_S(S/\m S).$$
Hence the result follows.
\end{proof}

We need the following important consequence of \autoref{prop:independengofsop} and \autoref{lem:limCMflatextension}.

\begin{corollary}
\label{cor:limCMmultiplicitylargerthengenerators}
Let $(R,\m)$ be a local ring of dimension $d$ and let $\{M_n\}$ be a weakly lim Cohen--Macaulay sequence. Then
$$e(R)\geq \limsup_{n\to\infty}\frac{e(\m, M_n)}{\nu_R(M_n)} \geq \liminf_{n\to\infty}\frac{e(\m, M_n)}{\nu_R(M_n)}\geq1.$$
\end{corollary}
\begin{proof}
The first inequality follows since we can map $R^{\oplus \nu_R(M_n)}$ onto $M_n$ and thus $e(\m, M_n)\leq e(R)\cdot \nu_R(M_n)$ for each $n$. The second inequality is obvious.

Now we prove the third inequality. Set $R'=R[t]_{\m R[t]}$ and $M_n'=M_n\otimes_RR'$. Since $R\to R'$ is a faithfully flat extension with $R'/\m R'$ a field, by \autoref{lem:limCMflatextension}, $\{M_n'\}$ is a weakly lim Cohen--Macaulay sequence over $R'$. It is clear that $e(\m, M)=e(\m R', M\otimes_RR')$ and $\nu_R(M)=\nu_{R'}(M\otimes_RR')$ for every finitely generated $R$-module $M$. Therefore by replacing $R$ by $R'$ and $\{M_n\}$ by $\{M_n'\}$, we may assume that $R$ has an infinite residue field. Thus there exists a system of parameters $\underline{z}=z_1,\dots,z_d$ of $R$ that is a minimal reduction of $\m$ (see \cite[Corollary 4.6.10]{BrunsHerzogCM}). Since $\{M_n\}$ is weakly lim Cohen--Macaulay, by \autoref{prop:independengofsop} we know that $\chi_1(\underline{z}, M_n)=o(\nu_R(M_n))$. Therefore
\[\liminf_{n\to\infty}\frac{e(\m, M_n)}{\nu_R(M_n)}=\liminf_{n\to\infty}\frac{e(\underline{z}, M_n)}{\nu_R(M_n)}=\liminf_{n\to\infty}\frac{\length_R(M_n/(\underline{z})M_n)}{\nu_R(M_n)}\geq 1. \qedhere\]
\end{proof}

Finally, we introduce lim Ulrich and weakly lim Ulrich sequences.

\begin{definition}
\label{def:limUlrich}
Let $(R,\m)$ be a local ring of dimension $d$. A sequence of finitely generated $R$-modules $\{U_n\}$ is called {\it lim Ulrich} (resp., {\it weakly lim Ulrich}) if it is lim Cohen--Macaulay (resp., weakly lim Cohen--Macaulay) and $$\lim_{n\to\infty}\frac{e(\m, U_n)}{\nu_R(U_n)}= 1.$$
\end{definition}

\begin{remark}
\label{rmk:limsup}
\begin{enumerate}
  \item We caution the readers that, unlike the lim Cohen--Macaulay property, we cannot replace $\nu_R(U_n)$ by $\rank_R(U_n)$ in \autoref{def:limUlrich} when $R$ is a domain.
  \item By \autoref{cor:limCMmultiplicitylargerthengenerators}, to check $\{U_n\}$ is (weakly) lim Ulrich, it is enough to show $\{U_n\}$ is (weakly) lim Cohen--Macaulay and that $\limsup\limits_{n\to\infty}\frac{e(\m, U_n)}{\nu_R(U_n)} \leq 1.$
\end{enumerate}

\end{remark}

The following is the main result of this section.

\begin{theorem}
\label{thm:limUlrichimpliesLech}
Let $(R,\m)\to (S,\n)$ be a flat local extension of local rings such that $R$ is a domain. Suppose $R$ admits a weakly lim Ulrich sequence $\{U_n\}$. Then $e(R)\leq e(S)$.
\end{theorem}
\begin{proof}
We first reduce to the case that $S/\m S$ has finite length following \cite[Lemma 2.2]{MaLechconjectureindimensionthree}. We replace $S$ by $\widehat{S}$ to assume $S$ is complete, and then we choose a minimal prime $Q$ of $\m S$ such that $\dim S/Q=\dim S/\m S$. Now $R\to S_Q$ is a flat local extension such that $S_Q/\m S_Q$ has finite length over $S_Q$. We have
$$\dim S \geq \dim S/Q + \dim S_Q = \dim S/\m S + \dim R = \dim S,$$
where the equalities follow from our choice of $Q$ and the dimension formula for flat local extensions \cite[Theorem 15.1]{Matsumura86}. It follows that $\dim S/Q + \dim S_Q=\dim S$, and as $S$ is complete (in particular excellent) we can invoke the localization formula for multiplicities (see \cite[Theorem 2.1]{MaLechconjectureindimensionthree}, which originates from \cite{Nagata62}) to see that $e(S_Q)\leq e(S)$. Therefore we can replace $S$ by $S_Q$ to assume that $S/\m S$ has finite length.

Since $R$ is a domain and $\{U_n\}$ is a weakly lim Ulrich sequence, we have:
\begin{eqnarray*}
  e(R) &=& \lim_{n\to\infty}\frac{e(\m, U_n)}{\rank_RU_n}=\lim_{n\to\infty}\frac{\nu_R(U_n)}{\rank_RU_n} \\
   &=& \lim_{n\to\infty}\frac{\nu_S(U_n\otimes_R S)}{\rank_RU_n}\leq \lim_{n\to\infty}\frac{e(\n, U_n\otimes_R S)}{\rank_RU_n}=e(S).
\end{eqnarray*}
Here the inequality above follows from \autoref{cor:limCMmultiplicitylargerthengenerators}, as $\{U_n\otimes_R S\}$ is a weakly lim Cohen--Macaulay sequence over $S$ by \autoref{lem:limCMflatextension}. To see the last equality, note that we have a map $R^{\oplus\rank_RU_n}\to U_n$ whose kernel and cokernel have dimensions strictly less than $\dim R$, so after tensoring with $S$, we obtain a map $S^{\oplus\rank_RU_n}\to U_n\otimes_RS$ whose kernel and cokernel have dimensions strictly less than $\dim S$. Hence we have $e(\n, U_n\otimes_R S)=e(\n, S^{\oplus\rank_RU_n})=e(S)\cdot \rank_RU_n.$
\end{proof}

We end this section with a proposition which follows from more general results in \cite{BhattHochsterMaLimCohenMacaulaysequence}. As this work is still in the stage of preparation, we give the proof of the proposition for the sake of completeness. This proposition also appeared as \cite[Lemma 5.20]{IyengarMaWalkerMultiplicity} where it was pointed out that the hypothesis $R$ is a domain is not necessary if we use $\nu_R(M_n)$ instead of $\rank_R(M_n)$ (the proof is essentially the same).

\begin{proposition}
\label{prop:localcohomologycriterion}
Let $(R,\m)$ be a local domain of dimension $d$ and let $\{M_n\}$ be a sequence of finitely generated modules of dimension $d$. Suppose $H_\m^j(M_n)$ has finite length for all $n$ and all $j<d$. Then $\{M_n\}$ is a lim Cohen--Macaulay sequence if $$\length_R(H_\m^j(M_n))=o(\rank_RM_n)$$
for all $j<d$.
\end{proposition}
\begin{proof}
Let $\underline{x}=x_1,\dots,x_d$ be a system of parameters of $R$. We have
$$H_i(\underline{x}, M_n)=H^{-i}(K_\bullet(\underline{x}, R)\otimes_R M_n)=H^{-i}(K_\bullet(\underline{x}, R)\otimes_R \mathbf{R}\Gamma_\m(M_n)),$$
where $H^{-i}(-)$ denotes $(-i)$-th cohomology of the complex. It follows that we have a spectral sequence:
$$H_{j+i}(\underline{x}, H_\m^j(M_n))\Rightarrow H_i(\underline{x}, M_n).$$
If $j=d$, then $j+i>d$ when $i\geq 1$. So for all $i\geq 1$ we have
\[\length_R(H_i(\underline{x}, M_n))\leq \sum_{j=0}^{d-1}\length_R\left(H_{j+i}(\underline{x}, H_\m^j(M_n))\right)\leq \sum_{j=0}^{d-1}2^d\cdot \length_R(H_\m^j(M_n))=o(\rank_RM_n),\]
where the second inequality holds because the $(i+j)$-th term of the Koszul complex has rank $\binom{d}{i+j}\leq 2^d$. This completes the proof that $\{M_n\}$ is lim Cohen--Macaulay by \autoref{lem:limCMrank}.
\end{proof}

\section{Main result for graded rings}

In this section we prove our main results. We first recall that if $R_1,\dots, R_c$ are $\mathbb{N}$-graded rings over a field $k$ and $M_i$ is a $\mathbb{Z}$-graded module over $R_i$ for each $1\leq i\leq c$, then the {\it Segre product} of $M_1, M_2, \dots, M_c$ is defined as $$M_1\# M_2 \# \cdots \# M_c := \oplus_{j\in\mathbb{Z}} [M_1]_j\otimes [M_2]_j \otimes\cdots \otimes [M_c]_j.$$

\begin{setting}
\label{setting}
We fix an infinite field $k$ of characteristic $p>0$ and let $q=p^e$ (eventually we will let $e\to\infty$ so one should think of $q$ as being very large). We consider
$$W_q^c:=k[x_1, y_1]\# k[x_2, y_2](q)\# \cdots \# k[x_c, y_c]((c-1)q),$$
which is a rank one module over the ring $$T_c=k[x_1, y_1]\# k[x_2, y_2]\# \cdots \# k[x_c, y_c].$$ We note that $T_c$ is a standard graded ring of dimension $c+1$: the degree $j$ part is spanned by monomials whose total degree in $x_i$ and $y_i$ is $j$ for each $1\leq i\leq c$. Hence $T_c$ is module-finite over $A_c=k[z_1,z_2,\dots,z_{c+1}]$ where $z_1,\dots,z_{c+1}$ are degree one elements in $T_c$ that form a homogeneous system of parameters of $T_c$. We will view $W_q^c$ as a graded module over $A_c$ that lives in non-negative degrees (because $k[x_1,y_1]$ only lives in non-negative degrees). We abuse notations and let $\m$ denote the homogeneous maximal ideal of $A_c$. Since $W_q^c$ is torsion-free and reflexive, we have $H_\m^0(W_q^c)=H_\m^1(W_q^c)=0$.
\end{setting}

The next lemma on the degrees and dimensions of local cohomology modules of $W_q^c$ is elementary. In fact, since $W_q^c$ is explicitly described, precise dimensions of each degree of its local cohomology modules can be computed. We are not interested in the precise formulas so we state and prove what we need. Geometrically, this corresponds to the sheaf cohomology of $O_{\mathbb{P}^1}(t)\boxtimes O_{\mathbb{P}^1}(t+q) \boxtimes \cdots \boxtimes O_{\mathbb{P}^1}(t+(c-1)q)$ on a product of projective lines when $t, q$ vary. It is worth pointing out that our construction is closely related to the notion of {\it supernatural vector bundles}, see \cite[Section 6]{EisenbudSchreyerJAMS} for more general constructions.

\begin{lemma}
\label{lem:localcohomologyofW_q}
With notation as in \autoref{setting}, we have
\begin{enumerate}[(1)]
  \item $H_\m^j(W_q^c)$ is nonzero only in degrees $-(j-2)q-2,\dots,-(j-2)q-q$ for each $2\leq j\leq c$.
  \item $H_\m^{c+1}(W_q^c)$ is nonzero only in degrees $\leq -(c-1)q-2$.
  \item Fix an integer $r\geq 1$. For each $2\leq j\leq c+1$, as $q\to \infty$ one has
  $$\dim_k H_\m^{j}(W_q^c)_{-(j-2)q-r} =o(q^c).$$
   Moreover for each $t\geq 0$ one has
   $$\dim_k H_\m^{c+1}(W_q^c)_{-(c+t)q-r}=o(q^{c+1}).$$
\end{enumerate}
\end{lemma}
\begin{proof}
We use induction on $c$, the case $c=1$ is obvious. Now suppose the lemma is proven for $c-1$. Since $W_q^c=W_q^{c-1}\#k[x_c,y_c]((c-1)q)$, it follows from the Kunneth formula for local cohomology (see \cite{GotoWatanabeOngradedringsI}) that
\begin{equation}
\label{eqn:localcohmologyofW_q}
\begin{cases}
 \hspace{0.5em}   H_\m^j(W_q^c)=H_\m^j(W_q^{c-1})\# \left(k[x_c,y_c]((c-1)q)\right),  \mbox{ for all } j\leq c \\
 \hspace{0.5em}   H_\m^{c+1}(W_q^c)=H_\m^n(W_q^{c-1})\# H_\m^2(k[x_c,y_c]((c-1)q)).  
\end{cases}
\end{equation}
Note that we are ignoring terms that are $0$ coming from the inductive hypothesis when applying the Kunneth formula. From \autoref{eqn:localcohmologyofW_q}, parts $(1)$ and $(2)$ are clear by the inductive hypothesis. 

To establish part $(3)$, we note that by \autoref{eqn:localcohmologyofW_q} and the induction hypothesis, for $j\leq c$,
\begin{eqnarray*}
  \dim_k H_\m^j(W_q^c)_{-(j-2)q-r} &=& \dim_kH_\m^j(W_q^{c-1})_{-(j-2)q-r}\cdot \dim_k (k[x_c,y_c])_{(c+1-j)q-r} \\
   &=& o(q^{c-1})\cdot((c+1-j)q-r+1)=o(q^c).
\end{eqnarray*}
For the top local cohomology, again by \autoref{eqn:localcohmologyofW_q} and the induction hypothesis,
\begin{eqnarray*}
  \dim_k H_\m^{c+1}(W_q^c)_{-(c+t)q-r} &=& \dim_kH_\m^c(W_q^{c-1})_{-(c+t)q-r}\cdot \dim_kH_\m^2(k[x_c,y_c])_{-(t+1)q-r}  \\
   &=& o(q^{c})\cdot((t+1)q+r-1).
\end{eqnarray*}
This gives $o(q^{c})$ for $t=-1$ and $o(q^{c+1})$ for $t\geq0$.
\end{proof}

The following immediate consequence is what we will need in the sequel. We adopt the following notation: if $M$ is a $\mathbb{Z}$-graded module, then $M_{a \mod q}:=\oplus_{i\in\mathbb{Z}} M_{a+iq}$

\begin{corollary}
\label{cor:lowerlocalcohomologyofW_q}
With notation as in \autoref{setting}, for any fixed positive integer $r$ and any $0\leq j\leq c$, we have
$$\dim_k H_\m^j(W_q^c)_{-r \mod q}=o(q^c) \text{ as } q\to\infty.$$
\end{corollary}
\begin{proof}
This follows directly from parts $(1)$ and $(3)$ of \autoref{lem:localcohomologyofW_q}.
\end{proof}

Now we state and prove our main result on weakly lim Ulrich sequences. Recall that a field $k$ of positive characteristic $p>0$ is called {\it $F$-finite} if $[k^{1/p}:k]<\infty$.

\begin{theorem}
\label{thm:MainLimUlrich}
Let $(R,\m)$ be a standard graded domain over an infinite $F$-finite field $k$ of characteristic $p>0$ localized at the homogeneous maximal ideal. Then $R$ admits a weakly lim Ulrich sequence.
\end{theorem}
\begin{proof}
Let $d=\dim R$. If $d=0$ then $R$ is a field, so $R$ is an Ulrich module. When $d=1$ it is easy to see that $\m^N$ is an Ulrich module for $N\gg0$. In the rest of the proof we will assume $d\geq 2$.

Since $R$ is standard graded and $k$ is infinite, there exists homogeneous degree one elements $z_1,\dots,z_d$ of $R$ that form a minimal reduction of $\m$. We identify the subring $A:=k[z_1,\dots,z_d]$ with the ring $A_{d-1}$ as in \autoref{setting}. Thus we have a sequence of finitely generated modules $\{W_q^{d-1}\}$ over $A$ where $q=p^e$. We will show that the following sequence:
$$U_e:=F_*^e\left((R\otimes_AW_q^{d-1})_{-1 \mod q}\right)$$
is a weakly lim Ulrich sequence over $R$.

Note that the $R$-module structure on $U_e$ is well-defined: under the $e$-th Frobenius pushforward $F^e_*(-)$, $x\in R$ acts as $x^q$ so elements in $R\otimes_AW_q^{d-1}$ of degree $\equiv -1 \mod q$ are preserved under the $R$-action. Also note that we take the degree $\equiv -1 \mod q$ in the definition of $U_e$ just for simplicity: in fact the proof will show that any {\it fixed negative integer} $-r$ will work.

We briefly outline the proof strategy when $R$ is Cohen--Macaulay, in which case we will show that $\{U_e\}$ is lim Ulrich (but see also \autoref{rmk:limUlrich}). In this case, $U_e$ is a direct sum of Frobenius pushforward of shifted copies of $(W_q^{d-1})_{-1 \mod q}$ and so $\length_R(H_\m^j(U_e))<\infty$ for all $j<d$. Thus we can apply \autoref{prop:localcohomologycriterion}, which allows us to compute local cohomology over $A$ and invoke \autoref{cor:lowerlocalcohomologyofW_q}. By the choice of $A$, the multiplicity computed over $A$ and over $R$ are the same, and it is easy to compare the rank over $A$ and over $R$. Putting these together will show $\{U_e\}$ is lim Cohen--Macaulay. To show $\{U_e\}$ is lim Ulrich, we estimate $\nu_R(U_e)$ by computing the dimension of certain graded piece of $W_q^{d-1}$ that cannot be in $\m U_e$.


\subsection*{The case $R$ is Cohen--Macaulay} Since $R$ is Cohen--Macaulay and is a graded module-finite extension of the polynomial ring $A$, we know $R\cong \oplus_{i=1}^s A(-a_i)$ as a graded $A$-module where $s=\rank_AR$ and $a_i\geq 0$ for each $i$. Thus we have
$$U_e\cong\oplus_{i=1}^s F^e_*(W_q^{d-1}(-a_i)_{-1 \mod q})\cong\oplus_{i=1}^s F^e_*((W_q^{d-1})_{-1-a_i \mod q}) $$
as graded $A$-modules. Recall that $W_q^{d-1}$ is a rank one module over $T_{d-1}$, and
$$\dim_k (T_{d-1})_t=(t+1)^{d-1}=t^{d-1}+o(t^{d-1}),$$
thus the multiplicity of $T_{d-1}$ as an $A$-module is $(d-1)!$. Therefore
$$e(\underline{z}, W_q^{d-1})=e(\m_A,W_q^{d-1})=(d-1)!.$$
It follows that $\rank_AW_q^{d-1}=(d-1)!$. We claim the following:
\begin{claim}
\label{clm:veroneserank}
For every fixed negative integer $-r$, the rank of $(W_q^{d-1})_{-r \mod q}$ as a module over $A^{(q)}$, the $q$-th Veronese subring of $A$, is equal to $(d-1)!$.
\end{claim}
\begin{proof}[Proof of Claim]
Since we have $\rank_AW_q^{d-1}=(d-1)!$, it is enough to show that
\begin{equation}
\label{eqn:veroneserank}
(W_q^{d-1})_{-r \mod q}\otimes_{A^{(q)}}\text{Frac}(A)=W_q^{d-1}\otimes_A\text{Frac}(A).
\end{equation}
To see this, first note that one containment $\subseteq$ is obvious. Next, every homogeneous element of the right hand side of \autoref{eqn:veroneserank} can be written as $\frac{w}{x}$, where $w\in W_q^{d-1}$ and $x\in A$ are homogeneous elements. Since $A$ is generated over $k$ by degree one forms, we can pick $y\in A$ such that $\deg w+\deg y\equiv -r \mod q$. It follows that $wy\in (W_q^{d-1})_{-r \mod q}$, and hence $\frac{w}{x}=\frac{wy}{xy} \in (W_q^{d-1})_{-r \mod q}\otimes_{A^{(q)}}\text{Frac}(A)$. This proves the other containment $\supseteq$ of \autoref{eqn:veroneserank}.
\end{proof}
By \autoref{clm:veroneserank}, we know the rank of $F^e_*((W_q^{d-1})_{-r \mod q})$ over $A^{(q)}$ is equal to $(d-1)!q^{d+\alpha}$ where $\alpha=\log_p[k:k^p]$. Therefore, since $\rank_{A^{(q)}}A=q$, for every fixed negative integer $-r$, we have
\begin{equation}
\label{eqn:rankofF^eW_q}
\rank_AF^e_*((W_q^{d-1})_{-r \mod q})=(d-1)!q^{d+\alpha-1}.
\end{equation}

To show $\{U_e\}$ is lim Cohen--Macaulay, by \autoref{prop:localcohomologycriterion} it is enough to prove that for every fixed negative integer $-r$ and each $j\leq d-1$,
\begin{equation}
\label{eqn:lengthoflocalcohomology}
\length_A\left(H_\m^j(F^e_*((W_q^{d-1})_{-r \mod q}))\right) =o\left(\rank_AF^e_*((W_q^{d-1})_{-r \mod q})\right)=o(q^{d+\alpha-1}).
\end{equation}
But since $H_\m^j(F^e_*((W_q^{d-1})_{-r \mod q}))=F^e_*(H_\m^j(W_q^{d-1})_{-r \mod q})$ and under the Frobenius pushforward $F^e_*$, the lengths get multiplied by $p^\alpha$, \autoref{eqn:lengthoflocalcohomology} follows from \autoref{cor:lowerlocalcohomologyofW_q}.

Finally, to show $\{U_e\}$ is lim Ulrich, we note that
\begin{eqnarray*}
  e(\m, U_e)=e(\underline{z}, U_e) &=&\sum_{i=1}^{s}e(\m_A, F^e_*((W_q^{d-1})_{-1-a_i \mod q})) \\
   &=& \sum_{i=1}^{s}\rank_AF^e_*((W_q^{d-1})_{-1-a_i \mod q})=(d-1)!sq^{d+\alpha-1}
\end{eqnarray*}
by \autoref{eqn:rankofF^eW_q}. On the other hand, since $R\otimes_AW_q^{d-1}$ lives in non-negative degrees, $\m^{[q]}\cdot (R\otimes_AW_q^{d-1})$ lives in degree $\geq q$. Therefore by the definition of $U_e$, we know that
$$\nu_R(U_e)\geq \dim_kF^e_*\left((R\otimes_AW_q^{d-1})_{q-1}\right).$$
However, by the definition of $W_q^{d-1}$ as in \autoref{setting}, for every fixed negative integer $-r$, we know that
\begin{equation}
\label{eqn:dimensiondeg(q-r)}
\dim_k(W_q^{d-1})_{q-r}=(q-r+1)(2q-r+1)\cdots ((d-1)q-r+1)=(d-1)!q^{d-1} + o(q^{d-1}).
\end{equation}
Therefore, since $a_i\geq 0$, we have
$$\dim_kF^e_*\left((R\otimes_AW_q^{d-1})_{q-1}\right) = \sum_{i=1}^{s} \dim_kF^e_*\left((W_q^{d-1})_{q-1-a_i}\right)=(d-1)! sq^{d+\alpha-1} + o(q^{d+\alpha-1}). $$
Putting the above together, we have
$$\limsup_{e\to\infty}\frac{e(\m, U_e)}{\nu_R(U_e)}\leq \limsup_{e\to\infty}\frac{e(\m, U_e)}{\dim_k F^e_*\left((R\otimes_AW_q^{d-1})_{q-1}\right)}=1.$$
Therefore by \autoref{rmk:limsup} (b), $\{U_e\}$ is a lim Ulrich sequence.

\subsection*{The general case} To handle the general case we first observe that our argument in the Cohen--Macaulay case proves that for every fixed negative integer $-r$, $F^e_*((W_q^{d-1})_{-r \mod q})$ is a lim Cohen--Macaulay sequence over $A$ (see \autoref{eqn:rankofF^eW_q}, \autoref{eqn:lengthoflocalcohomology}, and \autoref{prop:localcohomologycriterion}). In particular, we have (dropping $F^e_*$ results in dividing the vector space dimensions by $q^{\alpha}$)
\begin{equation}
\label{eqn:W_qweaklylimCM}
\dim_k\left(\frac{(W_q^{d-1})_{-r \mod q}}{(\underline{z}^q)(W_q^{d-1})_{-r \mod q}}\right)=(d-1)!q^{d-1}+o(q^{d-1}).
\end{equation}
On the other hand, we know that $\dim_k(W_q^{d-1})_{q-r}=(d-1)!q^{d-1}+o(q^{d-1})$ by \autoref{eqn:dimensiondeg(q-r)} and that $(W_q^{d-1})_{q-r} \cap (\underline{z}^q)(W_q^{d-1})_{-r\mod q}=0$ for degree reason (recall that $W_q^{d-1}$ only lives in non-negative degrees). This together with \autoref{eqn:W_qweaklylimCM} imply that
\begin{equation}
\label{eqn:dimensionneq(q-r)}
\dim_k\left(\left({W_q^{d-1}}/{(\underline{z}^q)W_q^{d-1}}\right)_{-r \mod q, \neq q-r}\right)=o(q^{d-1}).
\end{equation}

We now prove that $\{U_e\}$ is a weakly lim Cohen--Macaulay sequence. Let $s=\rank_AR$. We have a degree-preserving short exact sequence
\begin{equation}
\label{eqn:shortexactsequence}
0\to \oplus_{i=1}^sA(-b_i)\to R\to C\to 0
\end{equation}
where $C$ has dimension less than $d$ (note that $b_i\geq 0$ for all $i$). The rank of $U_e$ over $A$ is the same as the rank of
$$F^e_*\left(((\oplus_{i=1}^sA(-b_i))\otimes_AW_q^{d-1})_{-1\mod q}\right)\cong\oplus_{i=1}^s F^e_*((W_q^{d-1})_{-1-b_i \mod q})$$
over $A$. Therefore by \autoref{eqn:rankofF^eW_q}, we still have $$\rank_AU_e=e(\underline{z}, U_e)=(d-1)!sq^{d+\alpha-1}.$$ Thus to show $\{U_e\}$ is weakly lim Cohen--Macaulay, it is enough to show $\length_R(U_e/(\underline{z})U_e)\leq (d-1)!sq^{d+\alpha-1} +o(q^{d+\alpha-1})$ by \autoref{lem:limCMequivalentforms} (applied to $\underline{x}=\underline{z}$). Dropping $F^e_*$, this comes down to prove that
\begin{equation}
\label{eqn:wantedlenghsestimate}
\dim_k\left(\frac{(R\otimes_AW_q^{d-1})_{-1 \mod q}}{(\underline{z}^q)(R\otimes_AW_q^{d-1})_{-1 \mod q}}\right)\leq (d-1)!sq^{d-1}+o(q^{d-1}).
\end{equation}
From \autoref{eqn:shortexactsequence}, we obtain an exact sequence:
$$\frac{\oplus_{i=1}^s (W_q^{d-1})_{-1-b_i \mod q}}{(\underline{z}^q)(\oplus_{i=1}^s (W_q^{d-1})_{-1-b_i \mod q})}\to \frac{(R\otimes_AW_q^{d-1})_{-1 \mod q}}{(\underline{z}^q)(R\otimes_AW_q^{d-1})_{-1 \mod q}}\to \frac{(C\otimes_AW_q^{d-1})_{-1 \mod q}}{(\underline{z}^q)(C\otimes_AW_q^{d-1})_{-1 \mod q}}\to 0.$$
By \autoref{eqn:W_qweaklylimCM}, in order to establish \autoref{eqn:wantedlenghsestimate} it is enough to show that
$$\dim_k\left(\frac{(C\otimes_AW_q^{d-1})_{-1 \mod q}}{(\underline{z}^q)(C\otimes_AW_q^{d-1})_{-1 \mod q}}\right)=o(q^{d-1}).$$
Since $C$ is a finitely generated graded $A$-module of dimension less than $d$ and lives in non-negative degrees, it has a graded filtration by $(A/P_i)(-u_i)$, where $P_i$ are nonzero homogeneous prime ideals of $A$ and $u_i\geq 0$. So it is enough to prove that for any fixed homogeneous prime ideal $P\subseteq A$ and any $u\geq 0$, we have
$$\dim_k\left(\frac{(W_q^{d-1})_{-1-u \mod q}}{(PW_q^{d-1})_{-1-u \mod q}+(\underline{z}^q)(W_q^{d-1})_{-1-u \mod q}}\right)=o(q^{d-1}).$$

At this point, we invoke \autoref{eqn:dimensionneq(q-r)}. Thus in order to establish the above, it is enough to show that
$$\dim_k({W_q^{d-1}}/{PW_q^{d-1}})_{q-1-u}=o(q^{d-1}).$$
Fix $0\neq z\in P$ of degree $a>0$. Since ${W_q^{d-1}}/{zW_q^{d-1}}\twoheadrightarrow {W_q^{d-1}}/{PW_q^{d-1}}$ and $W_q^{d-1}$ is torsion-free, we know that
\begin{eqnarray*}
  \dim_k({W_q^{d-1}}/{PW_q^{d-1}})_{q-1-u} &\leq& \dim_k({W_q^{d-1}}/{zW_q^{d-1}})_{q-1-u} \\
   &=&\dim_k (W_q^{d-1})_{q-1-u}-\dim_k(W_q^{d-1})_{q-1-u-a}=o(q^{d-1})
\end{eqnarray*}
where the last equality follows from \autoref{eqn:dimensiondeg(q-r)}. This completes the proof of \autoref{eqn:wantedlenghsestimate} and hence we have established that $\{U_e\}$ is weakly lim Cohen--Macaulay.

Finally, we prove that $\{U_e\}$ is weakly lim Ulrich. Again since $R\otimes_AW_q^{d-1}$ only lives in non-negative degrees, $\m^{[q]}\cdot (R\otimes_AW_q^{d-1})$ lives in degree $\geq q$. Thus by the definition of $U_e$, we know that
$$\nu_R(U_e)\geq \dim_kF^e_*\left((R\otimes_AW_q^{d-1})_{q-1}\right).$$
Thus it remains to show that
\begin{equation}
\label{eqn:wanteddimensionestimate}
\dim_k(R\otimes_AW_q^{d-1})_{q-1}\geq (d-1)!sq^{d-1}+o(q^{d-1}),
\end{equation}
because this then implies that $\dim_kF^e_*\left((R\otimes_AW_q^{d-1})_{q-1}\right)\geq (d-1)!sq^{d+\alpha-1}+o(q^{d+\alpha-1})$ while $e(\m, U_e)=e(\underline{z}, U_e)=(d-1)!sq^{d+\alpha-1}$. To establish \autoref{eqn:wanteddimensionestimate}, we need the following claim. 
\begin{claim}
\label{clm:higherTorestimate}
Let $M$ be a finitely generated graded $A$-module which is nonzero only in non-negative degrees. Then for any fixed negative integer $-r$ and any $i\geq 1$, we have
$$\dim_k(\Tor_i^A(M, W_q^{d-1})_{-r \mod q})=\dim_k(H^{-i}(M\otimes^\mathbb{L}_AW_q^{d-1})_{-r \mod q})=o(q^{d-1}).$$
\end{claim}
\begin{proof}[Proof of Claim]
Since all the lower local cohomology modules of $W_q^{d-1}$ have finite length, $W_q^{d-1}$ is Cohen--Macaulay on the punctured spectrum of $A$. Since $A$ is regular, this means $W_q^{d-1}$ is finite free on the punctured spectrum of $A$ and hence $\Tor_i^A(M, W_q^{d-1})$ has finite length for all $i\geq 1$. A simple spectral sequence argument shows that
$$H^{-i}(M\otimes^\mathbb{L}_AW_q^{d-1})=H^{-i}(\mathbf{R}\Gamma_\m(M\otimes^\mathbb{L}_A W_q^{d-1}))=H^{-i}(M\otimes^\mathbb{L}_A\mathbf{R}\Gamma_\m (W_q^{d-1})) \text{ for all $i\geq 1$.}$$
As a consequence, we have a degree-preserving spectral sequence:
$$\Tor_{j+i}^A(M, H_\m^j(W_q^{d-1}))\Rightarrow H^{-i}(M\otimes^\mathbb{L}_AW_q^{d-1}).$$
Next we consider a minimal graded finite free resolution of $M$ over $A$:
\begin{equation}
\label{eqn:minimalfreeresolution}
0\to \oplus_l A(-a_{nl}) \to \cdots \to \oplus_l A(-a_{1l})\to \oplus_l A(-a_{0l})\to  0
\end{equation}
where $n=\text{pd}_AM$ and all the $a_{ij}$ are non-negative integers (since $M$ lives in non-negative degrees). If $j\leq d-1$, then using the above free resolution to compute $\Tor_{j+i}^A(M, H_\m^j(W_q^{d-1}))$, we see that
$$\dim_k(\Tor_{j+i}^A(M, H_\m^j(W_q^{d-1}))_{-r \mod q})\leq \sum_l \dim_k H_\m^j(W_q^{d-1})_{-r-a_{i+j,l} \mod q}=o(q^{d-1})$$
by \autoref{cor:lowerlocalcohomologyofW_q}. But if $j=d$, then $j+i\geq d+1$ so $\Tor_{j+i}^A(M, H_\m^j(W_q^{d-1}))=0$ since $A$ is regular of dimension $d$. Therefore all the $E_2$-contributions of $H^{-i}(M\otimes^\mathbb{L}_AW_q^{d-1})_{-r \mod q}$ have $k$-vector space dimensions $o(q^{d-1})$. This completes the proof of the claim.
\end{proof}

Now we return to the proof of the theorem, the short exact sequence \autoref{eqn:shortexactsequence} induces:
$$\Tor_1^A(C, W_q^{d-1})_{q-1}\to \oplus_{i=1}^sW_q^{d-1}(-b_i)_{q-1}\to (R\otimes_AW_q^{d-1})_{q-1}\to (C\otimes_AW_q^{d-1})_{q-1}\to 0.$$
It follows that
\begin{eqnarray*}
  \dim_k (R\otimes_AW_q^{d-1})_{q-1} &\geq & \sum_{i=1}^{s}\dim_k (W_q^{d-1})_{q-1-b_i}-\dim_k\Tor_1^A(C, W_q^{d-1})_{q-1}  \\
   &=& (d-1)!sq^{d-1}+o(q^{d-1})
\end{eqnarray*}
where the last equality follows from \autoref{eqn:dimensiondeg(q-r)} and \autoref{clm:higherTorestimate}. This completes the proof of \autoref{eqn:wanteddimensionestimate} and hence $\{U_e\}$ is a weakly lim Ulrich sequence, as desired.
\end{proof}

\begin{remark}
\label{rmk:limUlrich}
The sequence $\{U_e\}$ constructed in \autoref{thm:MainLimUlrich} is in fact lim Ulrich even if $R$ is not Cohen--Macaulay. This was recently established in our joint work with Iyengar and Walker \cite{IyengarMaWalkerMultiplicity}, using sheaf cohomology computations and then passing to the affine cones. Furthermore, we can prove the classes $\{[U_e]/\rank_R(U_e)\}$ converge to the class $[R]_d$ in the Grothendieck group $G_0(R)$, see \cite[Theorem 7.1]{IyengarMaWalkerMultiplicity}. Although these stronger results are not needed for our application to Lech's conjecture, they are crucial in our study of Betti numbers and Dutta multiplicities, see \cite{IyengarMaWalkerMultiplicity} for more details and explanations. 
\end{remark}


\begin{theorem}
\label{thm:MainLech}
Let $(R,\m)\to (S,\n)$ be a flat local extension of local rings. Suppose $(R,\m)$ is a standard graded ring over a perfect field $k$ localized at the homogeneous maximal ideal. Then $e(R)\leq e(S)$.
\end{theorem}
\begin{proof}
Since every minimal prime of $R$ is homogeneous, by the same argument as in \cite[Lemma 2.2]{MaLechconjectureindimensionthree}, we may assume $(R,\m)$ is a standard graded domain and $\dim R=\dim S$. We can further assume that $k$ is infinite and $F$-finite by replacing $R$ and $S$ by $R[t]_{\m R[t]}$ and $S[t]_{\n S[t]}$. The conclusion in characteristic $p>0$ now follows from \autoref{thm:limUlrichimpliesLech} and \autoref{thm:MainLimUlrich} (note that we only need to assume $k$ is $F$-finite).

Next we suppose $k$ has characteristic $0$ and $R\to S$ is a counter-example to the theorem. Then $\widehat{R}\to \widehat{S}$ is a flat local extension with $e(\widehat{R})>e(\widehat{S})$. Applying the argument in \cite[Lemma 5.1]{MaLechconjectureindimensionthree}, we may assume $k\cong R/\m\cong S/\n$ is algebraically closed and $\widehat{R}\to \widehat{S}$ is module-finite (note that $R$ is still standard graded over $k$). Now applying the reduction procedure in \cite[Subsection 5.1]{MaLechconjectureindimensionthree}\footnote{In \cite{MaLechconjectureindimensionthree}, we are not assuming $R$ is the completion of a finite type algebra therefore we choose a complete regular local ring $A$ inside $R$ and descend data to the Henselization of the localization of a polynomial ring, while here $R$ is finite type (in fact standard graded) over $k$ so we can run the same argument over $R$, the counter-example then descends to the Henselization of $R_\m$ and thus to a pointed \'{e}tale extension of $R_\m$.}, there exists a pointed \'{e}tale extension $R'$ of $R_\m$ and a finite flat extension $S'$ of $R'$ such that $e(R)=e(R')> e(S)$. But then by inverting elements if necessary, we may assume that we have
$$R\to R''=\left(\frac{R[x]}{f}\right)_g \to S''$$
such that $R''$ is standard \'e{t}ale over $R$ near a maximal ideal $\m''$ lying over $\m$, $R''\to S''$ is finite flat with a maximal ideal $\n''\in S''$ lying over $\m''$, and that $e(\m, R)=e(\m'', R'')>e(\n'', S'')$. We can reduce this set up to characteristic $p\gg0$ as in \cite[Subsection 5.2]{MaLechconjectureindimensionthree} to obtain
$$R_\kappa \to R''_\kappa \to S''_\kappa$$
with $\n''_\kappa$ a maximal ideal of $S''_\kappa$ lying over the homogeneous maximal ideal $\m_\kappa$ of $R_\kappa$, such that $(R_\kappa)_{\m_\kappa}\to (S''_\kappa)_{\n''_\kappa}$ is flat and $e((R_\kappa)_{\m_\kappa})>e((S''_\kappa)_{\n''_\kappa})$ (note that $R_\kappa \to R''_\kappa$ is always flat since $f$ is a monic polynomial in $x$). Thus we arrive at a counter-example (with $(R_\kappa, \m_\kappa)$ standard graded over an $F$-finite field $\kappa$) in characteristic $p>0$, which is a contradiction.
\end{proof}

\begin{remark}
Given \autoref{thm:MainLech}, one might try to consider attacking Lech's conjecture by passing to the associated graded rings. However, for a flat local extension $(R,\m)\to (S,\n)$, the induced map $gr_{\m}(R)\to gr_{\n}(S)$ of associated graded rings need not be flat (e.g., when $\m \subseteq \n^2$, then this map sends the homogeneous maximal ideal of $gr_{\m}(R)$ to zero). Therefore our \autoref{thm:MainLech} does not imply Lech's conjecture in general.
\end{remark}

Finally, we mention that in \cite{BhattHochsterMaLimCohenMacaulaysequence}, it is proven that every $F$-finite complete local domain of characteristic $p>0$ admits a lim Cohen--Macaulay sequence $\{F^e_*R\}$, which follows from standard methods in tight closure theory \cite{HochsterHunekeTC1}.

In an early version of this manuscript, we have asked whether every $F$-finite complete local domain of positive characteristic admits a lim Ulrich sequence. However, Yhee \cite{YheeUlrichDoNotExist} subsequently found an example of a two-dimensional complete local domain that does not admit even a weakly lim Ulrich sequence. Nevertheless, to the best of our knowledge, the following question remains open:

\begin{question}
\label{question:existence}
Does every $F$-finite complete normal (or Cohen--Macaulay) domain of characteristic $p>0$ admit a lim Ulrich sequence, or at least a weakly lim Ulrich sequence?
\end{question}

\bibliographystyle{skalpha}
\bibliography{CommonBib}
\end{document}